\documentclass[12pt, a4paper]{amsart}

\usepackage{a4wide}
\usepackage[english]{babel}
\usepackage[T2A]{fontenc}
\usepackage[utf8]{inputenc} 
\usepackage{amsfonts}
\usepackage{amssymb, amsthm, amscd}
\usepackage{amsmath}
\usepackage{mathtools}
\usepackage{needspace}
\usepackage{etoolbox}
\usepackage{lipsum}
\usepackage{comment}
\usepackage{cmap}
\usepackage[pdftex]{graphicx}
\usepackage[unicode]{hyperref}
\usepackage[matrix,arrow,curve]{xy}
\usepackage[usenames,dvipsnames]{xcolor}
\usepackage{colortbl}
\usepackage{textcomp}
\usepackage{cite}
\usepackage{euscript}

\DeclareFontFamily{U}{mathx}{}
\DeclareFontShape{U}{mathx}{m}{n}{<-> mathx10}{}
\DeclareSymbolFont{mathx}{U}{mathx}{m}{n}
\DeclareMathAccent{\widehat}{0}{mathx}{"70}
\DeclareMathAccent{\widecheck}{0}{mathx}{"71}

\makeatletter
\def\@settitle{\begin{center}
		\baselineskip14\p@\relax
		\bfseries
		\LARGE
		\@title
	\end{center}
}
\makeatother

\newtheorem{theorem}{Theorem}[section]

\newtheorem{lemma}[theorem]{Lemma}
\newtheorem{corollary}[theorem]{Corollary}
\newtheorem{remark}[theorem]{Remark}
\newtheorem{definition}[theorem]{Definition}
\newtheorem{example}[theorem]{Example}

\theoremstyle{definition}

\renewcommand{\leq}{\leqslant}
\newcommand{\R}{\mathbb{R}}
\newcommand{\N}{\mathbb{N}}
\newcommand{\Z}{\mathbb{Z}}
\newcommand{\F}{\mathcal{F}}

\newcommand{\K}{\mathcal{K}}
\newcommand{\G}{\mathcal{G}}
\newcommand{\LL}{\mathcal{L}}

\newcommand{\const}{\mathrm{const}}
\newcommand{\CC}{\mathbb{C}}

\DeclareMathOperator{\eps}{\varepsilon}

\newenvironment{enumerate*}%
  {\begin{enumerate}%
    \setlength{\itemsep}{1pt}%
    \setlength{\parskip}{1pt}}%
  {\end{enumerate}}

\makeatletter
\newcommand*{\@old@slash}{}\let\@old@slash\slash
\def\slash{\relax\ifmmode\delimiter"502F30E\mathopen{}\else\@old@slash\fi}
\makeatother

\title{Universal frame set for rational functions}
\author{Andrei V. Semenov}

\thanks{The work was carried out with the financial support of the Ministry of Science and Higher Education of the Russian Federation in the framework of a scientific project under agreement no. 075-15-2025-013. Author is winner of the “Leader” competition conducted by the Foundation for the Advancement of Theoretical Physics and Mathematics “BASIS” and would like to thank its sponsors and jury.}

\address{Andrei V. Semenov
Saint Petersburg State University\\
Department of Mathematics and Computer Science\\
Russia. }

\email{asemenov.spb.56@gmail.com}

\begin{document}
\maketitle

\begin{abstract}
Let $g \in L^2(\R)$ be a rational function of degree $M$, i.e., there exist polynomials $P, Q$ such that $g = \frac{P}{Q}$ and $\deg P < \deg Q \leq M$. We prove that for any $\eps>0$ and any $M \in \mathbb{N}$, there exists a universal set $\Lambda \subset \R$ of upper Beurling density less than $1+\eps$ such that the system 
$$\left\{ e^{2\pi i \lambda t } g(t-n) \colon (\lambda, n) \in \Lambda \times \Z \right\}$$
 forms a frame in $L^2(\R)$ for any well-behaved rational function $g$.
\end{abstract}

\section{Introduction}

For $(\lambda, \mu) \in \R^2$, define a time-frequency shift operator $\pi_{\lambda, \mu} $ on $L^2(\R)$ by the rule
$$\pi_{\lambda, \mu} g (t):= e^{2\pi i \lambda t} g(t - \mu), \quad  g\in L^2(\R).$$
Now, for a fixed $g \in L^2(\R)$ and a countable set $\LL \subset \R^2$, we define a {\it Gabor system} $\G(g, \LL)$ as follows:
$$\G(g, \LL) := \{\pi_{\lambda, \mu} g \mid (\lambda, \mu) \in \LL\}.$$
The system $\G(g,\LL)$ is called a {\it Gabor frame} if for some constants $A, B >0$ one has
\begin{equation}
A\|f\|^2_2\leq \sum_{(\lambda, \mu) \in \LL}|\langle f, \pi_{\lambda, \mu}g \rangle|^2\leq B\|f\|^2_2, \text{ for any } f\in L^2(\mathbb{R}),
\label{frameineq}
\end{equation}
where $\langle \cdot, \cdot \rangle$ denotes the standard scalar product in $L^2(\R)$. Inequality (\ref{frameineq}) is a {\it frame inequality}.
\par
Recently, in \cite{BKL1}, a technique for the class of rational functions was developed in the classical case where $\LL = \alpha \Z \times \beta \Z$. We aim to construct a universal Gabor frame for rational functions by developing a similar technique for a non-classical configuration of the index set $\LL$.

\subsection{The case of simple poles}
Here we formulate our main result for the class of rational functions with only simple poles.

\begin{definition}\label{def:K}
Let $\K(N)$ be the class of rational functions $g \in L^2(\R)$ with simple poles and of degree $N$, i.e., of the form
 $$g(t) = \sum_{k=1}^{N} {{a_k} \over {t - i w_k}}, \text{ where } a_k \in \CC \text{ and } w_k \in \CC \setminus i\R, $$
 such that
\begin{equation}\label{eq:defK}
\sum\limits_{k=1}^N a_k e^{2\pi w_k t} \ne 0 \text{ for any } t <0. 
\end{equation}
\end{definition}
We normalize the Fourier transform as follows:
$$\widehat{f}(z) = \int_{\R} f(t) e^{2\pi i t z} dt.$$

\begin{remark}
Note that if $w_k>0$ for any $k$, then $\sum\limits_{k=1}^N a_k e^{-2\pi w_k t} \theta(t)$ is (up to a multiplicative constant) the Fourier transform of $g(t)$, where $\theta$ is the Heaviside function. Thus, condition (\ref{eq:defK}) in Definition \ref{def:K} can be rewritten simply as
$$\widehat{g}(t) \ne 0 \quad \text{for any } t >0.$$
\end{remark}
Note that for a set $\Lambda \subset \R$, its {\it upper Beurling density} $D(\Lambda)$ is defined by the formula
$$D(\Lambda) = \limsup_{a \to \infty} \sup_{R \in \R} {{\# \{x \in \Lambda \mid x \in [R, R+a]\}} \over {a}}.$$
Throughout the remainder of this paper, we shall refer to $D(\Lambda)$ simply as the {\it density} of $\Lambda$. We aim to prove the following universal result.\\
$\!$\par
{\bf Theorem 1}. {\it For any $\eps >0$ and any $N \in \N$, there exists a set $\Lambda = \Lambda(\eps, N) \subset \R$ with upper density $D(\Lambda) \leq 1+\eps$ such that the system
$$\G(g, \Lambda\times \Z) := \{e^{2\pi i \lambda t} g(t - n) \mid (\lambda, n) \in \Lambda \times \Z\}$$
forms a frame in $L^2(\R)$ for any rational function $g \in \K(N)$.}\par
$\!$\par
Note that the condition on $D(\Lambda)$ is sharp. Indeed, for such a system to form a frame, we must have $D(\Lambda)\ge 1$. So, the condition $D(\Lambda) \leq 1+\eps$ implies that $D(\Lambda)$ {\it is arbitrarily close to the critical density}.\par
As an interesting consequence, we obtain the following corollary:
\par\medskip
{\bf Corollary 1}. {\it For any $\eps >0$ and any $N \in \N$, there exists a set $\Lambda = \Lambda(\eps, N) \subset \R$ with density $D(\Lambda) \leq 1+\eps$ such that the system
$$\G(g, \Lambda\times \Z) := \{e^{2\pi i \lambda t} g(t - n) \mid (\lambda, n) \in \Lambda \times \Z\}$$
forms a frame in $L^2(\R)$ for any $g(t) = \sum_{k=1}^{N} {{a_k} \over {t - i w_k}} $ such that $w_k>0$ and $\widehat{g}(t) \ne 0$ on $(0, +\infty)$.}

\subsection{The general case}
In fact, a more general result holds. Using our technique, we prove the existence of a universal set for any well-behaved rational function, not only for those with simple poles. \par
To define the class $\K_1(M)$ of well-behaved rational functions, one should interpret the natural condition $\widehat{g}(t) \ne 0 \text{ for any } t >0$ in the case of poles of order greater than 1.\par

\begin{definition}\label{def:K2}
For any $M \in \mathbb{N}$, let $\K_1(M)$ be the class of rational functions of degree $M$, i.e., of the form
 $$g(t) = \sum_{k=1}^{N} {{a_k} \over {(t - i w_k)^{j_k}}}, \text{ where } a_k \in \CC,  w_k \in \CC \setminus i\R \text{  and  } \sum_{k=1}^N j_k = M,$$
 such that  
\begin{equation}\label{eq:defK2}
\sum_{k=1}^N a_k e^{2\pi w_k t} {{(2\pi i)^{j_k-1}} \over {(j_k-1)!}} t^{j_k-1} \ne 0 \text{ for any } t <0.
\end{equation}
\end{definition}

\begin{remark} 
In the case of simple poles, i.e., $j_1 = j_2 = \ldots = j_N = 1$ and $M= \sum_k j_k = N$, condition \eqref{eq:defK2} reduces to $\sum_{k=1}^N a_k e^{2\pi w_k t} \ne 0$ for any $t <0$. Thus, $\K(M) \subset \K_1(M)$.
\end{remark}
Moreover, if we have $w_k > 0$ for all $1 \leq k \leq N$, then both formulas (\ref{eq:defK}) and (\ref{eq:defK2}) are equivalent to the condition
$$\widehat{g}(t) \ne 0 \text{ for any } t >0.$$

Now we are ready to state our second main result.\\
$\!$\par
{\bf Theorem 2}. {\it For any $\eps >0$ and any $M \in \N$, there exists a set $\Lambda = \Lambda(\eps, M) \subset \R$ with density $D(\Lambda) \leq 1+\eps$ such that the system
$$\G(g, \Lambda\times \Z) := \{e^{2\pi i \lambda t} g(t - n) \mid (\lambda, n) \in \Lambda \times \Z\}$$
forms a frame in $L^2(\R)$ for any rational function $g \in \K_1(M)$.}\par

\subsection{State of the Art} 
For a given window function $g$, a fundamental problem in time-frequency analysis is to determine the configurations of the index set $\LL \subset \R^2$ for which the Gabor system $\G(g, \LL)$ forms a frame in $L^2(\R)$. 

The classical case $\LL = \alpha \Z \times \beta \Z$ for $\alpha, \beta >0$ is the most studied.  For example, a complete description of the frame set, i.e., the set of parameters $(\alpha, \beta) \in \R^2$ such that $\G(g, \alpha \Z \times \beta \Z)$ forms a frame in $L^2(\R)$, has been established for families of (Gaussian) totally positive functions of finite type (see \cite{Gro1, Gro2}), generalizing classical results \cite{L, S, Jans2, Jans3, JansStr}. More recently, full characterizations were obtained for rational functions of Herglotz type (see \cite{BKL1}), shifted sinc-functions (see \cite{BAVS}), and a wide class of strictly decreasing continuous functions supported on a semi-axis (see \cite{BK}). In all these lattice settings, a well-known necessary density condition applies, requiring the critical lattice density to satisfy $D(\LL) = 1/\alpha\beta \ge 1$.

When moving to non-classical index sets $\LL \subset \R^2$, the behavior of Gabor frames becomes significantly more intricate. The answer in the general case is not known even for the Cauchy kernel $g(t) = {{1} \over {t-iw}}$. The best known result in this case is a criterion for $\G({{1} \over {t-iw}}, \LL)$ to form a frame (see \cite{BKL2}) in the case of direct product configurations $\LL = L \times M$. \par 
A prominent line of research addresses the existence of Gabor frames near the critical density. In a landmark paper, Balan, Casazza, and Landau introduced a quantitative approach to the redundancy of infinite frames, proving that any $\ell^1$-localized frame with redundancy strictly greater than one contains a subframe with redundancy arbitrarily close to one (see \cite{Balan}). Recently, Bownik and van Velthoven \cite{Bownik} generalized this framework to general reproducing kernel Hilbert spaces on metric measure spaces using a selector form of Weaver's conjecture. Crucially, they settled open questions regarding the existence of Gabor frames near the critical density for nonlocalized windows, demonstrating that such frames can be obtained by thinnings of highly overcomplete systems.\par
Note that the results mentioned above are non-constructive and depend on the specific choice of the window function $g$. Therefore, it is of special interest to construct a \textit{single universal set} of near-critical density that simultaneously works for some class of well-behaved functions. 

\subsection{Structure of the paper}
In Section \ref{sect:Outline}, we give a complete outline of the paper, including sketches of the proofs of Theorem 1 and Theorem 2. 
In Section \ref{sect:Crit}, we develop the main criterion for the system $\G(g, \Lambda \times \Z)$ to form a frame in $L^2(\R)$. We present the case of simple poles separately for the sake of clarity. 
In Section \ref{sect:Lambda}, we construct the desired universal frame set $\Lambda$. 
In Section \ref{sect:Oper}, we reduce the problem to the invertibility of a certain matrix operator defined on $\ell^2(\Z)$ and study its inner structure. 
In Section \ref{sect:Lemma}, we deduce the main technical lemma, which is used in the proofs of both main theorems. 
Finally, in Sections \ref{sect:Th1} and \ref{sect:Th2}, we provide the proofs of Theorem 1 and Theorem 2, respectively.

\subsection{Notations}
We denote by $PW_a$ the Fourier image of $L^2(0,a)$. The notation $f^{(n)}(t)$ denotes the $n$-th derivative of a function $f$. Let us also denote by $\{x\}$ the fractional part of $x\in \R$ and by $[x]$ the largest integer which is strictly smaller than $x$. By $\# A$ we denote the cardinality of a set $A$. We use the notation $M_{n,m}(\CC)$ for the space of all complex-valued matrices with $n$ rows and $m$ columns. In the case where $n=m$, we simply write $M_n(\CC)$. \par
Finally, by $U(z) \lesssim V(z)$ (equivalently $U(z) \gtrsim V(z)$), we mean that there exists a constant $C>0$ such that $U(z) \leq C V(z)$ holds for all $z$ in the set under consideration. We write $U(z) \asymp V(z)$ if both $U(z) \lesssim V(z)$ and $U(z) \gtrsim V(z)$ hold.\par
Despite the fact that $\K(N) \subset \K_1(N)$, we retain both symbols for the sake of clarity in notation.

\section{Outline of the paper}\label{sect:Outline}
Fix $M \in \mathbb{N}$ and let $g \in \K_1(M)$. First, we establish a criterion for the system $\G(g, \Lambda \times \Z)$ to form a frame in $L^2(\R)$. In Section 3 we define a collection of functions $\{m_s\}_{s=0}^{M-1}$ (see Definition \ref{def:m}) and prove the following
\begin{theorem}\label{main:crit} For any $g \in \K_1(M)$ and any $\Lambda = \{\lambda_m \mid m \in \Z\} \subset \R$ the system $\G(g, \Lambda \times \Z)$ forms a frame in $L^2(\R)$ if and only if for any $G \in L^2(\R)$ we have
\begin{equation}\label{eq:maincrit}
\|G\|_2^2 \asymp \sum_{m\in \Z} \int_0^1 \left|\sum_{s=0}^{M-1} G(t+\lambda_m + s) m_s(t) \right|^2 dt.
\end{equation}
\end{theorem}
Since the functions $\{m_s \}_{s = 0}^{M-1}$ are bounded uniformly in [0,1], the upper bound in Equation (\ref{eq:maincrit}) always holds. So, we only need to verify the lower bound.

\medskip

Let $\xi \in [0,1)$ and $G \in L^2(\R)$. For almost every $\xi \in [0,1)$ we have $\{G(\xi+n)\}_{n \in \Z} \in \ell^2(\Z)$. Now we can consider $G$ as a collection of sequences from $\ell^2(\Z)$ indexed by $\xi$:
$$G \mapsto \{G(\xi+n)\}_{n \in \Z}, \quad \xi \in [0,1).$$
We have 
$$\|G\|_{L^2(\R)}^2 = \int_0^1 \| \{G(\xi+n)\}_{n \in \Z} \|_{\ell^2 (\Z) }^2 d\xi. $$
For a fixed $\xi \in [0,1)$ we are looking for elements from the sequence $\{G(\xi+n)\}_{n\in \Z}$ on the right-hand side of (\ref{eq:maincrit}). This leads us to the following construction. Consider the equation
$$t + \lambda_k = \xi \mod 1, \quad t \in [0,1).$$
Observe that for a fixed $\xi \in [0,1)$ and for a fixed $k \in \Z$ there is only one solution $t_k \in [0,1)$ of such an equation. Define $L_{\xi} \colon L^2(\R) \to \ell^2(\Z)$ for any $\xi \in [0,1)$ by the formula
\begin{equation}\label{def:L}
L_{\xi} G = \left\{\sum_{s=0}^{M-1} G(t_k+\lambda_k) m_s(t_k)\right\}_{k \in \Z}.
\end{equation}
Since for any $k \in \Z$ we have $t_k+\lambda_k = \xi +n$ for some $n \in \Z$, the right-hand side depends only on the sequence $\{G(\xi+n)\}_{n\in \Z}$. Hence, we can consider $L_{\xi}$ as an operator from $\ell^2(\Z)$ to $\ell^2(\Z)$. From Theorem \ref{main:crit} it follows that $\G(g, \Lambda \times \Z)$ forms a frame if and only if for some $c>0$
\begin{equation}\label{eq:opercrit}
    \|L_{\xi} x\|^2 \ge c \|x\|^2 \text{ for any } x \in \ell^2(\Z)
\end{equation}
uniformly with respect to almost all $\xi \in [0,1)$.

\subsection{The operator $L_{\xi}$}
In Section \ref{sect:Lambda}, we construct the desired universal set $\Lambda$. For a fixed $\xi \in [0,1)$ and $n \in \Z$, consider the equation 
\begin{equation}\label{eq:mod}
    \xi + n = t+\lambda_k,
\end{equation}
for $t\in [0,1)$ and $\lambda_k \in \Lambda$. We are looking for the rows which corresponds to a given number $\xi + n$. By the construction of $\Lambda$, there is always a solution $\lambda_k$ to Equation \eqref{eq:mod} and the number of solutions is uniformly bounded. So, there are two regimes in the matrix of $L_{\xi}$:
$$I = \begin{pmatrix} \star & \star & \ldots & \star  \\
\vdots & \vdots && \vdots \\
 \star & \star & \ldots & \star
\end{pmatrix},  \quad II =
\begin{pmatrix}
 \star & \star & \ldots & \star  & 0 & 0 & \ldots  &0 \\
 0 & \star & \star & \ldots & \star  & 0 & \ldots  &0\\
 &   & \ddots &   \ddots & & \ddots& &\vdots  \\
0  &  \ldots & 0 & \star &  \star  & \ldots& \star & 0\\
0  & \ldots & 0 & 0 & \star & \star  & \ldots& \star
\end{pmatrix}
$$
The first regime corresponds to the case when there is more than one solution to Equation (\ref{eq:mod}), while the second regime corresponds to the case when there is only one solution. Hence, the matrix of $L_{\xi}$ consists of submatrices of the form:
\begin{equation}\label{eq:matrix}
\newcommand*{\tempb}{\multicolumn{1}{|c}{0}}
\newcommand*{\tempv}{\multicolumn{1}{|c}{\vdots}}
\begin{pmatrix} 
\star & \star & \ldots & \star & \tempb &  0  & 0&\ldots  &0\\
\star & \star & \ldots & \star  & \tempb  & 0 & 0 & \ldots  & 0\\
\vdots & \vdots&  & \vdots &   \tempv& \vdots  &  \vdots  & & \vdots \\
\star & \star & \ldots & \star   & \tempb & 0 &  0 & \ldots  &0 \\
\cline{1-4}
0 & \star & \star & \ldots & \star  & 0 & 0 & \ldots  &0 \\
0 & 0 & \star & \star & \ldots & \star  & 0 & \ldots  &0\\
\vdots & &   & \ddots &   \ddots & & \ddots& &\vdots  \\
0 & \ldots  & 0 & 0 & \star &  \star  & \ldots& \star & 0\\
0 & \ldots  & 0 & 0 & 0 & \star & \star  & \ldots& \star\\
\end{pmatrix}.
\end{equation}

Informally, this means that the first regime helps the frame inequality, while the second regime corresponds to ordinary rows in the matrix. We discuss the inner structure of $L_{\xi}$ more closely in Section \ref{sect:Oper}.

\subsection{Proofs of the Theorems}
We use Theorem 1 in the proof of the more general Theorem 2, so first we prove Theorem 1 in Section \ref{sect:Th1}. We reduce the problem of invertibility of $L_{\xi}$ for almost every $\xi \in [0,1)$ to the problem of uniform invertibility of submatrices of $L_{\xi}$ defined in \eqref{eq:matrix}. Next, we directly compute the determinants using the main technical lemma from Section \ref{sect:Lemma} and obtain the desired invertibility. \par 
In order to prove Theorem 2 we fix an arbitrary $g \in \K_1(M)$ and approximate it by a function with simple poles $g_{\eps} \in \K(M)$, using the following remark:
\begin{remark}\label{rem:asymp}
For any $f\in C^{\infty}(\R)$ and any sufficiently small $\eps>0$ we have
$$f^{(n)} (t) \sim {{1} \over {\eps^n}} \sum_{k=0}^n (-1)^{n-k} \binom{n}{k} f(t+k\eps),$$
where the strict meaning of $\sim$ is discussed later in Section \ref{sect:Th2}.
\end{remark}
It turns out that the operators $L_{\xi}(g)$ and $L_{\xi}(g_{\eps})$ are close enough (here we mean that $L_{\xi}(g)$ is constructed from $g$ and $L_{\xi}(g_{\eps})$ is constructed from $g_{\eps}$). We approximate the determinants of $L_{\xi}(g)$ by those of $L_{\xi}(g_{\eps})$, which are studied in Section \ref{sect:Th1}. Finally, using Theorem 1 and the invertibility of $L_{\xi}(g_{\eps})$ for almost every $\xi \in [0,1)$ we deduce the desired invertibility of $L_{\xi}(g)$.

\section{Main criterion}\label{sect:Crit}
In this section we prove Theorem \ref{main:crit}. We note that the technique developed here is similar to the proofs of Theorem 1.1 in \cite{BKL1} and Theorem 1.2 in \cite{BAVS}. So we use the notation from \cite{BKL1} and \cite{BAVS}. \par
To facilitate understanding of the paper, we split the section into two parts: the case of simple poles and the general case. In fact, the case of simple poles is a direct corollary of the general case. However, we want to state this separately in order  to provide an illustrative example free of heavy combinatorial technicalities. 

\subsection{The case of simple poles}
Fix $N \in \mathbb{N}$. Let $g\in \K(N)$ have the form $g(t) = \sum \limits_{k=1}^N {{a_k} \over {t-iw_k}}$ for some $a_k \in \CC$ and $w_k \in \CC \setminus i\R$. Fix a function $f\in L^2(\R)$ and let $\Lambda = \{\lambda_n \mid n \in \Z\}$ be (any) countable subset of $\R$ for a while. One may assume $\lambda_n \leq \lambda_{n+1}$ for any $n \in \Z$. We define $\pi_{m,n}  g(t) := e^{2\pi i \lambda_m t} g(t-n)$ for any $m, n \in \Z$. Hereinafter we use the symbol $\pi_{m,n} $ in the sense defined above.\par
Now, in order to show that $\G(g; \Lambda \times \Z)$ forms a frame, one needs to prove that
$$\| f\|_2^2 \asymp \sum_{(m,n) \in \Z \times \Z}  | \langle \pi_{m,n}  g, f \rangle|^2.$$
By duality we have
$$\|f\|_2 \asymp \left( \sum_{(m,n) \in \Z^2}  |\langle \pi_{m,n}  g, f\rangle |^2 \right)^{1/2}= \sup\limits_{c \in \ell^2(\Z^2): \|c\|=1}   \left| \sum_{(m,n) \in \Z^2} c_{m,n} \langle \pi_{m,n}  g, f\rangle \right|,$$
 Assume for now that $f \in C^{\infty}(\R)$ has compact support. Now we have
$$\sup\limits_{c \in \ell^2(\Z^2): \|c\|=1}  \left| \sum_{(m,n) \in \Z^2} c_{m,n} \langle \pi_{m,n}  g, f \rangle \right| = \sup\limits_{c \in \ell^2(\Z^2): \|c\|=1}  \left|  \sum_{(m,n) \in \Z^2} c_{m,n} \int_{\R} e^{2\pi i \lambda_m t} \sum_{k=1}^N {{a_k} \over {t - i w_k - n}} \overline{f(t)} dt \right|.$$

Following \cite{BKL1} and \cite{BAVS}, define a function 
\begin{equation}\label{def:hm}
h_m(z) := (1-e^{2\pi i z}) \sum_{n \in \Z} {{c_{m,n}} \over {z-n}}.
\end{equation}
We have $\widecheck{h_m}(t) =\const \cdot \sum_{n \in \mathbb{Z}} c_{m,n} e^{-2\pi i n t} \in L^2(0,1)$, since $\{c_{m,n}\}_{n \in \Z} \in \ell^2(\Z)$. Hence, $h_m \in PW_1$. Now put $P(t) = \prod_{j=1}^N (1-e^{2\pi i (t-iw_j)})$ and $p_j(t) = 1-e^{2\pi i (t-iw_j)}$. Put $P_k(t) = {{P(t)} \over {p_k(t)}} = \prod_{k \ne j} (1-e^{2\pi i (z-iw_j)})$ and observe that for any $k\in \Z$ one has
\begin{equation}\label{eq:A}
P_k(t) = \sum_{l=0}^{N-1} A_{k,l} e^{2\pi i l t}, \text{ where } A_{k,l} = (-1)^l \sum_{\substack{\{j_1, \dots, j_l\} \subset \{1, \dots, N\} \setminus \{k\} \\ j_1 < \dots < j_l}} e^{2\pi (w_{j_1} + \ldots + w_{j_l})}.
\end{equation}
It is easy to see that
$$\sum_{n \in \Z} {{c_{m,n}} \over {t-iw_k-n}} = {{P_k(t)} \over {P(t)}} h_m (t-iw_k).$$
Now we have
$$ \|f\|_2 \asymp \sup\limits_{c \in \ell^2(\Z^2): \|c\|=1}  \left| \sum_{m \in \Z} \int_{\R} \frac{\overline{f(t)} e^{2\pi i \lambda_m t}}{\prod_{k=1}^N (1 - e^{2\pi i(t - iw_k)})}  \sum_{k=1}^N a_k P_k(t) h_m (t-iw_k) dt \right|.$$

Define $F$ by the rule 
$$F(t) = \frac{f(t)}{\prod_{k=1}^N (1 - e^{-2\pi i(t + i\overline{w}_k)})}.$$
 It is clear $F \in L^2(\R)$. Moreover, $\|F\|_2 \asymp \|f\|_2$, since the denominator is bounded above and below by some positive constant. For a fixed $m \in \Z$ we obtain
$$\int_{\R} \overline{F(t)e^{-2\pi i \lambda_m t}} \sum_{k=1}^N a_k P_k(t) h_m (t-iw_k) dt = \int_{\R} \overline{F(t)e^{-2\pi i \lambda_m t}} \sum_{l=0}^{N-1} e^{2\pi i lt} M_l(t) dt,$$ 
where 
$$M_l(t) = \sum_{k=1}^N a_k A_{k,l} h_m(t - iw_k).$$ 
Define $G$ as the inverse Fourier transform of $\overline{F}$ and observe $\|f\|_2 \asymp \|G\|_2$. Put 
\begin{equation}
\label{eq:m}
m_l(t) = \sum\limits_{k=1}^N a_k A_{k, l} e^{2\pi t w_k}.
\end{equation}
By Parseval's identity we have
$$ \|f\|_2 \asymp \sup\limits_{c \in \ell^2(\Z^2): \|c\|=1}  \left| \sum_{m \in \Z} \int_{\R} \overline{F(t)e^{-2\pi i \lambda_m t}} \sum_{l=0}^{N-1} e^{2\pi i lt} M_l(t) dt \right| =$$
$$= \sup\limits_{c \in \ell^2(\Z^2): \|c\|=1}  \left| \sum_{m \in \Z} \int_0^1 \left( \sum_{l=0}^{N-1} G(x + \lambda_m + l) m_l(x) \right) \cdot \widecheck{h_m} (x) dx \right|.$$
As the supremum is taken over all sequences $\{c_{m,n}\}$ on the unit sphere of $\ell^2(\Z \times \Z)$, the sequence of functions $\{\widecheck{h_m}\}_{m\in \Z}$ runs through the sphere, meaning that $\sum_m \|\widecheck{h_m}\|^2_2 = \const$, where the constant depends on the definition of the Fourier transform. By choosing 
$$\widecheck{h_m} (x) :=\sqrt{\const} \cdot \frac{ \overline{ \sum_{l=0}^{N-1} G(x + \lambda_m + l) m_l(x)} }{\left( \sum_{k\in \Z} \left\|\sum_{l=0}^{N-1} G(x + \lambda_k + l) m_l(x)   \right\|_2^2 \right)^{1/2}  }$$
we obtain
$$\|G\|_2 \asymp \|f\|_2 \gtrsim  \left( \sum_{m \in \Z} \int_0^1 \left| \sum_{l=0}^{N-1} G(x + \lambda_m + l) m_l(x)  \right|^2 dx \right)^{1/2}.$$

So we have established the criterion for a set $\Lambda$ to generate a frame in the case when all poles are simple.

\begin{theorem}\label{simple:crit} For any $g \in \K(N)$ and any $\Lambda = \{\lambda_n \mid n \in \Z\} \subset \R$ the system $\G(g; \Lambda \times \Z)$ forms a frame in $L^2(\R)$ if and only if there exists $C_1, C_2>0$ such that for any $G \in L^2(\R)$ we have
\begin{equation}\label{eq:crit}
C_1 \|G\|_2^2 \leq \sum_{m \in \Z} \int_0^1 \left| \sum_{l=0}^{N-1} G(t + \lambda_m + l) m_l(t)  \right|^2 dt \leq C_2 \|G\|_2^2.
\end{equation}
\end{theorem}
\begin{proof} The lower bound has already been proven, while  the upper bound can be obtained using the Cauchy-Schwarz inequality via the fact that all the $m_l(x)$ are uniformly bounded from above on $[0,1]$.
\end{proof}

\subsection{The general case}
In this subsection we prove Theorem \ref{main:crit}. Now we deal with the class $\K_1(M)$ for $M \in \mathbb{N}$. The window function $g$ now has the form
$$g(t) = \sum_{k=1}^N {{a_k} \over {(t-iw_k)^{j_k}}}, \quad a_k \in \CC, w_k \in \CC \setminus i\R \text{  and  } \sum_{k=1}^N j_k = M.$$

We may assume $1 \leq j_1 \leq \ldots \leq j_N$. Define $p_k(t):= \left(1 - e^{2\pi i (t-iw_k)}\right)^{j_k}$ and let $P$ be the product $P = \prod_{k=1}^N p_k$. Also define $P_k (t)= P(t)/p_k(t) =\prod_{l\ne k} (1-e^{2\pi i (t - iw_l)})^{j_l} $. As in the case of simple poles the function $F = f/\overline{P}$ lies in $L^2(\R)$ and we have $\|F\|_2 \asymp \|f\|_2$. \par
Using the technique developed in Subsection 3.1, we obtain
\begin{multline}\label{eq:prelim} 
\left( \sum_{n,m \in \Z} |(f, \pi_{m,n}  g)|^2 \right)^{1/2} \asymp \\
\sup_{c \in \ell^2(\Z^2): \|c\|=1} \left| \sum_{m \in \Z} \int_{\R} \overline{F(t) e^{-2\pi i \lambda_m t}} \cdot \sum_{k=1}^N a_k P_k(t) g_{m, j_k} (t-iw_k) dt \right|,
\end{multline}

\subsection{The trick}
Fix $m \in \Z$ for a while. Using the function $h_m$ defined in Equation (\ref{def:hm}), one can check that
$$ \sum_n {{c_{m,n}} \over {(z - n)^k}} = {{(-1)^{k-1}} \over {(k-1)!}} \sum_{l=0}^{k-1} \binom{k-1}{l} h_m^{(k-l-1)}(z)  \left( {{1} \over {1-e^{2\pi i z}}} \right)^{(l)}. $$
Put $f_l(z) = \left( {{1} \over {1-e^{2\pi i z}}} \right)^{(l)}$ and note that $f_l$ is a {\it rational function} of the variable  $e^{2\pi iz}$ with maximal degree $l+1$ in the denominator. It is clear that for any integer $k\ge 1$ we have 
$$h_m^{(k)} (z) =  (-1)^k k! (1-e^{2\pi i z}) \sum_n {{c_{m,n}} \over {(z-n)^{k+1}}} - \sum_{l=0}^{k-1} \binom{k}{l}  (2\pi i)^{k-l} e^{2\pi i z} \left(\sum_n {{c_{m,n}} \over {z-n}} \right)^{(l)}.$$
We multiply both sides by $(1-e^{2\pi i z})^k (-1)^k /k!$ and perform a shift $k \mapsto k-1$, which proves the following Lemma.
\begin{lemma} For any $k \in \mathbb{N}$ and any $m \in \Z$ we obtain
\begin{equation}\label{eq:trick}
\begin{split}
g_{m,k} (z) &= h_m^{(k-1)} (z) {{(e^{2\pi i z} -1)^{k-1}} \over {(k-1)!}} + \\
&+e^{2\pi i z} {{(e^{2\pi i z} -1)^{k-1}} \over {(k-1)!}} \sum_{l=0}^{k-2} (2\pi i)^{k-l-1} \binom{k-1}{l} \sum_{j=0}^l \binom{l}{j} h_m^{(l-j)} (z) f_j(z).
\end{split}
\end{equation}
For $k=1$, the second term on the right-hand side of \eqref{eq:trick} vanishes.
\end{lemma}

\begin{remark}\label{rem:pol}
The equation (\ref{eq:trick}) represents a {\it polynomial} formula in the variables $h_m^{(l)}(z)$ for $0 \leq l \leq k-1$ since the denominators in the second summand have their degrees at most $k-1$ and so they are being eliminated by $(e^{2\pi i z} -1)^{k-1}$. Moreover, this formula is $\mathrm{polynomial}$ with respect to the variables $e^{2\pi i z l}$ for $0\leq l \leq k-1$.
\end{remark}

\begin{example} For $k=1$ we trivially have $g_{m,1} = h_m$ from \eqref{def:hm}. For $k=2$ we obtain the formula
$$g_{m,2}(z) = -(1-e^{2\pi i z}) h_m'(z) -2\pi i \cdot e^{2\pi i z} h_m(z).$$
\end{example}

\subsection{The criterion}
Combining together equations (\ref{eq:prelim}) and (\ref{eq:trick}) we obtain
$$\left( \sum_{m,n\in\mathbb{Z}}|\langle f,\pi_{m,n}g\rangle|^2 \right)^{1/2}\asymp \sup\limits_{c \in \ell^2(\Z^2): \|c\|=1}  \left| \sum_m \int_{\mathbb{R}} \overline{F(t)e^{-2\pi i\lambda_m t}} \sum_{k=1}^N a_k P_k(t) \cdot \Psi_{m,k}(t) dt \right|,$$
where $\Psi_{m,k}(t)$ is given by the expansion
$$\Psi_{m,k}(t) = h_m^{(j_k-1)}(t-iw_k)\frac{(e^{2\pi i(t-iw_k)}-1)^{j_k-1}}{(j_k-1)!} + $$
$$+e^{2\pi i (t-iw_k)} {{\left(e^{2\pi i (t-iw_k)} -1\right)^{j_k-1}} \over {(j_k-1)!}}  \cdot \sum_{l=0}^{j_k-2} (2\pi i )^{j_k-l-1} \binom{j_k-1}{l} \sum_{b=0}^l \binom{l}{b} h_m^{(l-b)} (t-iw_k) f_b(t-iw_k).$$

Now we need to combine all functions with multipliers $e^{2\pi i d t}$ together. Fix $k \in \N$ for a while. By Remark \ref{rem:pol} one can choose the coefficients $a_{l,d}$ such that
\begin{equation}\label{eq:g}
g_{m,j_k} (t-iw_k) = \sum_{l,d=0}^{j_k-1} a_{l,d} h_m^{(l)} (t-iw_k)  e^{2\pi i t d}.
\end{equation}
The coefficients clearly depend on $k$, but to simplify the notation we do not write the index $k$ here. Observe that we have
\begin{equation}\label{eq:Pk}
P_k(t) = \prod_{l \ne k} \sum_{s=0}^{j_l} (-1)^s \binom{j_l}{s} e^{2\pi i t s} e^{2\pi w_l s } \stackrel{def}{=}  \sum_{s=0}^{M-j_k} A_{k, s} e^{2\pi i t s},
\end{equation}
where the coefficients $A_{k,s}$ are defined in the preceding equation. Combining Equations (\ref{eq:g}) and (\ref{eq:Pk}) together, we obtain
\begin{multline}\label{eq:Pg}
P_k(t) g_{m, j_k}(t-iw_k) = \sum_{l,d=0}^{j_k-1} \sum_{s=0}^{\sum_{l \ne k} j_l} a_{l,d} A_{k,s} h_m^{(l)} (t-iw_k) e^{2\pi i d t} e^{2\pi i s t} = \\
=\sum_{s=0}^{M-1} \sum_{l=0}^{j_k-1} B^{(k)}_{s,l} h_m^{(l)}(t-iw_k) e^{2\pi i st}
\end{multline}
for some coefficients $B_{s,l}^{(k)}$, where $M = \sum_{k=1}^N j_k$ is the absolute constant. Finally,
$$\sup\limits_{c \in \ell^2(\Z^2): \|c\|=1}  \left| \sum_{m} \int_{\R} \overline{F(t) e^{-2\pi i \lambda_m t}} \cdot \sum_{k=1}^N a_k \sum_{s=0}^{M-1} \sum_{l=0}^{j_k-1} B^{(k)}_{s,l} h_m^{(l)}(t-iw_k) e^{2\pi i st} dt \right| = $$
$$ = \sup\limits_{c \in \ell^2(\Z^2): \|c\|=1}  \left| \sum_{m} \int_{\R} \overline{F(t) e^{-2\pi i \lambda_m t}}  \sum_{s=0}^{M-1}  e^{2\pi i s t} \cdot \bigg( \sum_{k=1}^N \sum_{l=0}^{j_k-1} a_k B_{s,l}^{(k)} h_m^{(l)} (t-iw_k) \bigg) dt \right|.$$
Fix $m \in \Z$ and define $M_s(t)$ by the formula
\begin{equation}\label{eq:Mst}
M_s(t) =  \sum_{k=1}^N \sum_{l=0}^{j_k-1} a_k B_{s,l}^{(k)} h_m^{(l)} (t-iw_k).
\end{equation}
Observe that $h_m^{(l)}$ here varies from $h_m^{(0)}$ to $h_m^{(j_N-1)}$. Denote by $G$ the conjugate of the inverse Fourier transform of $F$ and note that $\|f\|_2 \asymp \|F\|_2 \asymp \|G\|_2$. By the Parseval-Plancherel theorem one has

$$ \sup\limits_{c \in \ell^2(\Z^2): \|c\|=1}  \left| \sum_{m} \int_{\R} \sum_{s=0}^{M-1} \overline{F(t) e^{-2\pi i (\lambda_m + s) t}} \cdot M_s(t) dt \right| = $$
$$=\sup\limits_{c \in \ell^2(\Z^2): \|c\|=1}  \left| \sum_m \int_0^1 \sum_{s=0}^{M-1} G(t+\lambda_m + s) \sum_{k=1}^N \sum_{l=0}^{j_k-1} a_k B_{s,l}^{(k)} \F^{-1}\left(h_m^{(l)} (x-iw_k)\right) (t) dt \right|,$$
where $\mathcal{F}^{-1}$ stands for the inverse Fourier transform. One can check that 
$$\F^{-1}\left(h_m^{(l)} (x-iw_k)\right) (t)  = (2\pi i )^l t^l e^{2\pi w_k t} \widecheck{h_m}(t).$$
\begin{definition} \label{def:m}
For a fixed natural $1\leq s \leq M$ we define 
$$m_s(t) = \sum_{k=1}^N \sum_{l=0}^{j_k-1} a_k B_{s,l}^{(k)} e^{2\pi w_k t} (2\pi i)^l t^l.$$
\end{definition}

\begin{remark}\label{rem:gen}
Observe that if $j_1= \ldots = j_N = 1$, then $m_s$ is equal to those defined in Equation \eqref{eq:m}, since $g_{m,1} = h_m$ by definition and $B_{s, 0}^{(k)} = A_{k, s}$ from Equation \eqref{eq:A}. So, this definition precisely generalizes the functions $m_s$ from the simple poles case.
\end{remark}

Now we have 
$$\|f\|_2 \asymp \sup\limits_{c \in \ell^2(\Z^2): \|c\|=1}  \left| \sum_m \int_0^1 \left(\sum_{s=0}^{M-1} G(t+\lambda_m + s) m_s(t) \right) \widecheck{h_m}(t) dt \right|.$$

As the supremum is taken over all sequences $\{c_{m,n}\}$ on the unit sphere of $\ell^2(\Z \times \Z)$, the sequence of functions $\{\widecheck{h_m}\}_{m\in \Z}$ runs through the sphere, meaning that $\sum_m \|\widecheck{h_m}\|^2_2 = \const$. So, as in the case of simple poles, by choosing 
$$\widecheck{h_m} (x) :=\sqrt{\const} \cdot \frac{ \overline{ \sum_{s=0}^{M-1} G(x + \lambda_m + s) m_s(x)} }{\left( \sum_{k\in \Z} \left\|\sum_{s=0}^{M-1} G(x + \lambda_k + s) m_s(x)   \right\|_2^2 \right)^{1/2}  }$$
we obtain
$$ \|G\|_2 \asymp \|f\|_2 \gtrsim \left( \sum_m \int_0^1 \left|\sum_{s=0}^{M-1} G(t+\lambda_m + s) m_s(t) \right|^2 dt \right)^{1/2}.$$
Hence, we obtain the lower bound, while the upper bound is delivered by the Cauchy-Schwarz inequality and the fact that all the $m_s$ are bounded uniformly from above on $[0,1]$.
\begin{corollary} In the case of simple poles we obtain Theorem \ref{simple:crit}.
\end{corollary}

\section{Construction of $\Lambda$}\label{sect:Lambda}
In this section, we construct the universal set $\Lambda$ from Theorem 1 and Theorem 2. The construction is the same for the case of simple poles and for the general case, so we fix $\eps>0$ and $M \in \mathbb{N}$ from Theorem 2 and construct the set $\Lambda = \Lambda(\eps, M)$.

\subsection{The construction}
Let $K=2M$ and $K_1 = \left[ {{K} \over {\eps}} \right]+1$. Place $K+1$ points of the form $\lambda_j = {{j} \over {K}}$, where $0 \leq j \leq K$. Now choose any $\delta \in \left(0, {{1} \over {K_1}} \right)$. Place one point $\lambda_{K+j}$ on each $[j, j+1]$ for any integer $1 \leq j \leq K_1$ by the rule:
\begin{equation}\label{eq:lambda}
\lambda_{K+j} = j+1-j\cdot \delta \text{ for any } 1 \leq j \leq K_1.
\end{equation}
We have $\lambda_{K+j} \in (j, j+1)$ since $\delta < 1/K_1$. Observe that there are exactly $K+1+K_1$ points chosen in $[0, K_1+1)$. \par 

Finally, consider a partition 
$$\R =\ldots\cup [-K_1-1,0) \cup [0, K_1+1) \cup [K_1+1, 2K_1+2) \cup \ldots$$
and apply this procedure to other intervals.\par
It is clear that
$$1< D(\Lambda) = \limsup_{a \to \infty} \sup_{R \in \R} {{\# \{x \in \Lambda \mid x \in [R, R+a]\}} \over {a}} = {{K+1+K_1} \over {K_1+1}} \leq 1+\eps.$$
Thus, we have constructed the set $\Lambda=\Lambda(\eps, M)$ of the desired density. Hereinafter, we fix $\Lambda = \{\lambda_n \mid n \in \Z\}$ such that $0 = \lambda_0$ and $\lambda_n < \lambda_{n+1}$ for any $n \in \Z$.

\section{Structure of $L_{\xi}$}\label{sect:Oper}
We use Theorem \ref{main:crit} to reduce the problem to the invertibility of the matrix operator $L_{\xi} \colon \ell^2(\Z) \to \ell^2(\Z)$, defined earlier in \eqref{def:L}. 
\subsection{The coefficients}
Choose $\xi \in [0,1)$. Following Section \ref{sect:Outline}, for a fixed $k \in \Z$ and $t \in (0,1)$ put 
$$b_k = t+\lambda_k - \xi  \in \Z.$$
Fix $k \in \Z$.  Let $[\lambda_k]=l$. If $\{\lambda_k\}\leq \xi$, then there exists $0 \leq  t_0 = \xi - \{\lambda_k\} \leq \xi$ such that $t_0+\{\lambda_k\} = \xi$. In this case we have $b_k = l = [\lambda_k]$. If $\{\lambda_k\}>\xi$, then there exists $\xi < t_1 = \xi+1-\{\lambda_k\}<1$ such that $t_1+\{\lambda_k\} = \xi+1$. So we have $b_k = l+1 = [\lambda_k]+1$. Hence, for $\zeta_k = \xi - \{\lambda_k\} \mod 1$ we define
\begin{equation} \label{eq:mlk}
a_l(k):= m_l(\zeta_k) = \begin{cases}
m_l(\xi - \{\lambda_k \}+1), & \xi - \{\lambda_k \}< 0,\\
m_l(\xi - \{\lambda_k \}), & \xi -\{ \lambda_k \}> 0.
\end{cases} 
\end{equation}
and in the $k$-th row of $L_{\xi}$, there are only $M$ possible non-zero coefficients $a_l(k)$ for $0 \leq l \leq M-1$. Now let us consider the inner structure of the operator.
\subsection{The blocks} 
Consider the equation $ t+ \lambda_n = \xi + b_n $ and assume $b_n = b_{n+1} \in \Z$ for some $n \in \Z$. In this case we have {\it a block}, i.e., two rows in $L_{\xi}$ of the first regime:
$$\begin{pmatrix} 
\ldots & 0 & a_0(n) & a_1(n) & \ldots  &a_{M}(n) &  0 & \ldots \\
\ldots & 0 &  a_0(n+1) & a_1(n+1) & \ldots  &a_{M}(n+1) &  0 & \ldots
\end{pmatrix}.$$
Hence, there exists $b \in \Z$ such that
\begin{equation}\label{eq:block}
t + \lambda_n = \xi + b \quad \text{ and } \quad t+\lambda_{n+1} = \xi +b.
\end{equation}
Now we have either $[\lambda_n] = [\lambda_{n+1}]$ and $\{\lambda_n\} < \{\lambda_{n+1}\} < \xi$ or $[\lambda_n] +1 =[\lambda_{n+1}]$ and $ \{\lambda_{n+1}\} < \xi < \{\lambda_n\} $. By induction we obtain
\begin{lemma}\label{lem:blocks} 
 For a fixed $\xi \in [0,1)$ the rows corresponding to  $\lambda_{j}, \ldots, \lambda_{j+k} \in \Lambda$ form a block in $L_{\xi}$ if and only if there exists $n\in \Z$ such that $\lambda_{j}, \ldots, \lambda_{j+k}  \in (\xi + n, \xi + n +1).$
\end{lemma}

\subsection{The segments}\label{sec:oper}
Now fix $\xi \in [0,1)$ and $\eps>0$ from Theorem 2. Set $M_1 = [2M/\eps]+1$. Let us take a closer look  at what happens in $[0,M_1+1)$. By Lemma \ref{lem:blocks}, the points from the arithmetic progression inside $[0,1]$ form two separate blocks: the first one with $\lambda_n \in [0,\xi]$ and the second one with $\lambda_n \in (\xi, 1]$:
\begin{equation}\label{twoblocks}
\newcommand*{\tempb}{\multicolumn{1}{|c}{0}}
\newcommand*{\templ}{\multicolumn{1}{c|}{0}}
\newcommand*{\tempv}{\multicolumn{1}{|c}{}}
\newcommand*{\templv}{\multicolumn{1}{c|}{\vdots}}
\newcommand*{\tstar}{\multicolumn{1}{|c}{\star}}
\newcommand*{\tvdots}{\multicolumn{1}{|c}{\vdots}}
\newcommand*{\tldots}{\multicolumn{1}{|c}{\ldots}}
\begin{pmatrix} 
\cline{2-5}
 \ldots & \tstar & \star & \ldots & \star & \tempb &   \ldots  \\
\ldots & \tstar & \star & \ldots & \star  & \tempb  & \ldots   \\
& \tvdots & \vdots&  & \vdots &   \tempv & \\
\ldots & \tstar & \star & \ldots & \star   & \tempb & \ldots   \\
\cline{2-6}
\ldots & \templ& \star & \star & \ldots & \star & \tldots  \\
\ldots & \templ& \star & \star & \ldots & \star  & \tldots  \\
& \templv& \vdots & \vdots&  & \vdots &   \tempv  \\
\ldots & \templ& \star & \star & \ldots & \star   &\tldots   \\
\cline{3-6}
\end{pmatrix}.
\end{equation}

On $[1,M_1+1)$ we have $\lambda_{2M+j} = j+1-j\delta$, so they form an $M$-diagonal subsegment of the second regime. Observe that $\{\lambda_{2M+j}\} = 1 - j\delta$ is strictly decreasing for any $1 \leq j \leq M_1$. Hence, there exists the index $j_0$ such that 
$$\{\lambda_{2M+j_0}\}  > \xi \quad \text{ and } \quad \{\lambda_{2M+j_0+1}\}  < \xi.$$
Now by Lemma \ref{lem:blocks} the rows $(j_0,j_0+1)$ form a $2\times M$-block inside the $M$-diagonal subsegment:
\begin{equation}\label{M-diag}
\newcommand*{\tstar}{\multicolumn{1}{|c}{\star}}
\newcommand*{\start}{\multicolumn{1}{c|}{\star}}
\begin{pmatrix}
 \star & \star & \ldots & \star  & 0 & 0 & \ldots  &0 & 0  \\
 0 & \star & \star & \ldots & \star  & 0 & \ldots  &0 & 0 \\
\vdots &   & \ddots &   \ddots & & \ddots& &   \vdots&  \vdots \\
\cline{4-7}
0  &  \ldots & 0 & \tstar &  \star  & \ldots& \start & 0 & 0\\
0  & \ldots & 0 &  \tstar & \star  & \ldots& \start & 0 & 0\\
\cline{4-7}
\vdots &   &  \vdots & & \ddots &   \ddots & & \ddots& \vdots \\
0 & \ldots & 0 &\ldots & 0 & \star &  \star & \ldots &\star 
\end{pmatrix}
\end{equation}
Finally, if $\lambda_{j_0}$ is the last point taken from $[0,M_1+1)$, then the $j_0$-th row forms a block with the next arithmetic progression inside $[M_1+1, M_1+2)$ (see construction of $\Lambda$, Section \ref{sect:Lambda}). \par
We define a {\it segment} to be the submatrix ranging from the block generated by points in $(\xi, 1)$ to the block generated by points in $[M_1+1, M_1+1+\xi]$. It consists of the starting block generated by $(\xi,1)$, followed by the $M$-diagonal subsegment together with the $2\times M$-block generated by $[1, M_1+1)$ and ends with the block generated by $[M_1+1,M_1+1+\xi]$. \par
Now the operator $L_{\xi}$ consists of segments of such a form, each of size $(2M+1+M_1)\times (M+M_1)$.

\section{Main lemma}\label{sect:Lemma}
Fix $N \in \mathbb{N}$. Let $\alpha > N$ be an arbitrary real number. Fix $\xi \in (0,1)$ such that $\xi > {{N-1}\over {\alpha}}$ and consider the matrix $B$ of the form
$$B = \begin{pmatrix}
m_0 (\xi) & m_1(\xi) & \ldots & m_{N-1}(\xi) \\
m_0 (\xi-{{1}\over {\alpha}}) & m_1(\xi-{{1}\over {\alpha}}) & \ldots & m_{N-1}(\xi-{{1}\over {\alpha}}) \\
\ldots & \ldots & \ldots & \ldots\\
m_0 (\xi-{{N-1}\over {\alpha}}) & m_1(\xi-{{N-1}\over {\alpha}}) & \ldots & m_{N-1}(\xi-{{N-1}\over {\alpha}}) 
\end{pmatrix} \in M_N(\CC),$$ 
where the functions $\{m_s(t)\}_{s=0}^{N-1}$ are defined in Equation \eqref{eq:m}. 
\begin{lemma} \label{lem:det} We have 
\begin{equation} \label{eq:det}
|\det B| = \prod_{k=1}^N |a_k| e^{2\pi \xi w_k} \cdot \prod_{1 \leq k<j \leq N} |e^{-2\pi {{w_k} \over {\alpha}}} - e^{-2\pi {{w_j} \over {\alpha}}}| \cdot \prod_{k <  l} |e^{2\pi w_k} - e^{2\pi w_l}|.
\end{equation}
\end{lemma}
This lemma can be easily deduced from the proof of Lemma 5.2 in \cite{BKL1}. However, for the sake of completeness of the article, we outline the proof.
\begin{proof}
Observe that $m_l\left(\xi - {{j} \over {\alpha}}\right) = \sum_{k=1}^N a_k A_{k, j} e^{2\pi \xi w_k} e^{-2\pi w_k {{j} \over {\alpha}}}$ for any $0 \leq j \leq N-1$. Now set 
$$A_k:=a_k e^{2\pi \xi w_k}, \quad y_k:=e^{-2\pi {{w_k} \over {\alpha}}}, \quad u_k = e^{2\pi w_k} \quad \text{ for any } 1 \leq k \leq N.$$
 It is now clear that 
$$B = \begin{pmatrix}
A_1 & A_2& \ldots & A_N\\
A_1 y_1 & A_2 y_2 & \ldots & A_N y_N\\
\ldots & \ldots & \ldots & \ldots\\
A_1 y_1^{N-1}& A_2 y_2^{N-1} & \ldots & A_N y_N^{N-1}\\
\end{pmatrix}  \cdot 
\begin{pmatrix}
1 & -\sum_{k\ne 1} u_k & \ldots & (-1)^{N-1} \prod_{k \ne 1} u_k \\
1 &  -\sum_{k\ne 2} u_k & \ldots &  (-1)^{N-1} \prod_{k \ne 2} u_k \\
\ldots & \ldots & \ldots & \ldots\\
1 &  -\sum_{k\ne N} u_k & \ldots & (-1)^{N-1} \prod_{k \ne N} u_k \\
\end{pmatrix},
$$
where the entries of the second matrix are the elementary symmetric polynomials in the variables $\{u_k\}_{k=1}^N$. Now denote the first matrix by $X$ and the second matrix by $Y$. It is clear that
$$\det Y = \pm \prod_{k < l} (u_k-u_l),$$
while 
$$\det X = \prod_{k=1}^N A_k \cdot \det \begin{pmatrix}
1 & 1& \ldots &1\\
y_1 & y_2 & \ldots & y_N\\
\ldots & \ldots & \ldots & \ldots\\
y_1^{N-1}& y_2^{N-1} & \ldots & y_N^{N-1}\\
\end{pmatrix} = \prod_{k=1}^N A_k \cdot  \prod_{i<j} (y_j- y_i).$$
Combining all the equalities above, we obtain the claim.
\end{proof}

\section{Proof of Theorem 1}\label{sect:Th1}
Now we are ready to prove our main result in the case of simple poles. Fix $N \in \mathbb{N}$ and $\eps > 0$ from Theorem 1 and fix $N_1 = [2N/\eps]+1$. Let us divide the proof into steps. \par
{\bf Step 1}. By Theorem \ref{main:crit}, it is sufficient to prove that
$$ \sum_{m \in \Z} \int_0^{1} \left| \sum_{l=0}^{N-1} G(t + \lambda_m + l) m_l(t)  \right|^2 dt \ge \const \cdot \|G\|_2^2 \quad \text{ for any } G \in L^2(\R).$$
 Fix $\xi \in [0,1)$ and construct the operator $L_{\xi}$ from Section \ref{sect:Oper}. Observe that
\begin{equation}\label{eq:mN}
m_{N-1}(t) = (-1)^{N-1} \cdot \sum_{k=1}^N a_k e^{2\pi w_k t} \prod_{j \ne k} e^{2\pi w_j} = (-1)^{N-1} e^{2\pi \cdot \sum_j w_j} \cdot \sum_{k=1}^N a_k e^{2\pi (t-1) w_k}.
\end{equation}
Now by Definition \ref{def:K} we have 
\begin{equation} 
m_{N-1}(t) \ne 0 \text{ for any } 0\leq t<1.
\end{equation}
 Consider the segments defined in Subsection \ref{sec:oper}. This segments provide a partition of the operator, so one can number them while preserving their order from top to bottom. Consider the finite-dimensional matrices $A_m$ corresponding to the $m$-th segment. Each $A_m$ has $2N+1+N_1$ rows and $N+N_1$ columns.\par

{\bf Step 2}. For any  $k \in \Z$ consider the arithmetic progression $\{\lambda_j = k(N_1+1)+{{j} \over {2N}} \}_{j=0}^{2N} \subset [kN_1+k, kN_1+k+1)$. This interval generates two blocks: 
\begin{itemize}
\item {\it the first type block}, generated by $\lambda_j \in [kN_1+k, kN_1+k+\xi)$,
\item {\it the second type block}, generated by $\lambda_j \in [kN_1+k+\xi, kN_1+k+1)$. 
\end{itemize}
One of these blocks consists of no fewer than $N+1$ rows. The type of this block is the same for each $k$ by construction of $\Lambda$. \par
 One can remove rows from $L_{\xi}$ since we only want to check the bottom inequality $\|L_{\xi}(x)\| \ge c \cdot \|x\|$ and removing rows can only decrease the left-hand side norm. If we have at least $N+1$ rows in the blocks {\it of the first type}, then delete all rows from the second-type blocks except the last one (which corresponds to the point $ kN_1+k+1$). Attach this row to the $N$-diagonal subsegment of the previous segment. Also, keep only the first $N+1$ rows from each block of the first type. \par
If we have at least $N+1$ rows in the blocks {\it of the second type}, then delete all the rows from the first type blocks except the first row. Attach this row to the $N$-diagonal subsegment of $A_{k-1}$. Also, keep only the last $N+1$ rows from each block of the second type. Hence, any block in $L_{\xi}$ is now an $(N+1)\times N$-matrix. \par
Finally, inside the $N$-diagonal subsegment, generated by $\big[kN_1+k+1, (k+1)N_1+k+1\big)$, we always have a $2\times N$-block, as shown in Subsection \ref{sec:oper}. Delete the row from this block which has the smaller last coefficient $|a_{N-1}(\cdot)|$. However, the number of rows inside $N$-diagonal subsegment remains unchanged, since we previously reduced one of the blocks inside $A_k$ to the one row.\par
Now each segment $A_m$ is an $(N+N_1+1)\times (N+N_1)$ matrix consisting of the block and the $N$-diagonal subsegment without inner blocks:
$$
\newcommand*{\tempb}{\multicolumn{1}{|c}{0}}
\newcommand*{\tempv}{\multicolumn{1}{|c}{\vdots}}
A_m = \begin{pmatrix} 
\star & \star & \ldots & \star & \tempb &  0  & 0&\ldots  &0\\
\star & \star & \ldots & \star  & \tempb  & 0 & 0 & \ldots  & 0\\
\vdots & \vdots&  & \vdots &   \tempv& \vdots  &  \vdots  & & \vdots \\
\star & \star & \ldots & \star   & \tempb & 0 &  0 & \ldots  &0 \\
\cline{1-4}
0 & \star & \star & \ldots & \star  & 0 & 0 & \ldots  &0 \\
0 & 0 & \star & \star & \ldots & \star  & 0 & \ldots  &0\\
\vdots & &   & \ddots &   \ddots & & \ddots& &\vdots  \\
0 & \ldots  & 0 & 0 & \star &  \star  & \ldots& \star & 0\\
0 & \ldots  & 0 & 0 & 0 & \star & \star  & \ldots& \star\\
\end{pmatrix}.
$$
Note that now every block of $L_{\xi}$ has {\it the same type}.

{\bf Step 3}. Fix $m\in \Z$ and consider the $m$-th segment with a matrix $A_m$. Define a function $A$ that maps $(\zeta_0, \ldots , \zeta_{N+N_1} )$ to a matrix
$$ \begin{pmatrix}
m_0(\zeta_0) & m_1(\zeta_0) & \ldots & m_{N-1}(\zeta_0) & 0 & \ldots & 0 \\
m_0(\zeta_1) & m_1(\zeta_1) & \ldots & m_{N-1}(\zeta_1) & 0 & \ldots & 0 \\
&  & \ldots & & \ldots &  &\\
m_0(\zeta_N) & m_1(\zeta_N) & \ldots & m_{N-1}(\zeta_N) & 0 & \ldots & 0 \\
0& m_0(\zeta_{N+1}) & \ldots & m_{N-2}(\zeta_{N+1}) &  m_{N-1}(\zeta_{N+1}) & \ldots \\
&  & \ldots & & \ldots &  &\\
0&\ldots & m_0(\zeta_{N+N_1}) & \ldots & \ldots &  m_{N-2}(\zeta_{N+N_1}) & m_{N-1}(\zeta_{N+N_1})  \\
\end{pmatrix}$$
where $m_l(\zeta)$ is defined in Equation (\ref{eq:m}). Denote the block of this matrix by $\Delta(\zeta_0, \ldots , \zeta_N)$. Clearly, for an $m$-th segment $A_m$, we have 
$$A_m = A(\{\lambda_{j_0}\}, \ldots , \{\lambda_{j_{N+N_1}}\} ), $$
where $\{j_0, \ldots, j_{N+N_1} \}$ is the set of row indices of the segment. 

{\bf Step 4}. By erasing one chosen row from each block, we may assume that each block has exactly $N$ rows. If all the blocks are  {\it of the first type}, then remove the last row of each block. If all the blocks are  {\it of the second type}, then remove the first row. Now all the blocks are $N\times N$ matrices and it is easy to see that
\begin{equation}
\det A(\zeta_0, \ldots , \zeta_{N+N_1})  = \det \Delta(\zeta_0 , \ldots , \zeta_N) \cdot \prod_{j=N+1}^{N+N_1} m_{N-1}(\zeta_j).
\label{matrix:det}
\end{equation}
Here we formally write $\Delta(\zeta_0 , \ldots , \zeta_N)$ as a function of $N+1$ variables, but we have already erased one row, so it is actually of $N$ variables.

{\bf Step 5}. In this step, we prove that $|a_{N-1}(j)| \ge \const > 0$ inside any $N$-diagonal subsegment of any segment $A_m$. The only possibility for $|a_{N-1}(j)|$ to be arbitrarily small is when the corresponding function $m_{N-1}(\zeta_j)$ has its argument arbitrarily close to 1 by \eqref{eq:mN}. If there exists a row $j$ inside the $N$-diagonal subsegment of $A_m$ such that $|a_{N-1}(j)|$ is sufficiently small, then $\xi$ is sufficiently close to $\{\lambda_j\}$ from the left by \eqref{eq:mlk}. By construction of $\Lambda$ we have
$$\{\lambda_j \}= 1-\delta j \quad  \text{ and } \quad \{\lambda_{j+1} \} = 1-\delta(j+1).$$
By Lemma \ref{lem:blocks} the rows $(j, j+1)$ form a $2\times N$-block. But in Step 2, we have already removed the row from such a block that has the smaller $|a_{N-1}(j)|$. The remaining $|a_{N-1}(j+1)|$ and $|a_{N-1}(j-1)|$ from neighboring rows cannot be arbitrarily small since $|\lambda_{j\pm1} - \lambda_j| =1 -\delta$. \par
If such a row is the last in the $N$-diagonal subsection, then it forms a block with the starting block of $A_{m+1}$ and is also removed in Step 2.\par
Hence, one can assume that there exists $C_1>0$ independent of $A_m$ such that in (\ref{matrix:det}) we have 
\begin{equation} \label{eq:det1}
\left| \prod_{j=N+1}^{N+N_1} m_{N-1}(\zeta_j) \right| \ge C_1^{N_1}
\end{equation}
for any segment of the operator. Now we only need to examine $\det \Delta (\zeta_0 \ldots , \zeta_N)$ for any arguments $\zeta_j$ corresponding to each segment.

{\bf Step 6}. Denote the block of the $m$-th segment by $\Delta_m$. If it is {\it of the first type}, for the $N$-tuple of arguments $(\zeta_0, \ldots , \zeta_{N-1})$ of the functions $\{m_s\}_{s=0}^{N-1}$ corresponding to $\Delta_m$ we have $\zeta_k = \xi - {{k} \over {2N}}$ for any $k$, since we erased other rows on Step 2.  Now $\Delta_m$ has the form
$$\begin{pmatrix} 
\sum\limits_{k=1}^N a_k e^{2\pi \xi  w_k} & \ldots  &  (-1)^{N-1} \sum\limits_{k=1}^N a_k e^{2\pi \xi w_k}  \cdot \prod\limits_{l \ne k} e^{2\pi w_l}\\
 \ldots &    &  \ldots\\
\sum\limits_{k=1}^N a_k e^{2\pi \left(\xi-{{N-1} \over {2N}} \right)w_k}  & \ldots  & (-1)^{N-1} \sum\limits_{k=1}^N a_k e^{2\pi \xi w_k}  e^{-2\pi w_k{{N-1} \over {2N}}}  \cdot \prod\limits_{l \ne k} e^{2\pi w_l} \\
\end{pmatrix}$$
If $\Delta_m$ is {\it of the second type}, then $\zeta_k = \xi - {{j} \over {2N}} +1$ for $N+1\leq j \leq 2N$. Hence, $\Delta_m$ has the form
$$\begin{pmatrix} 
\sum\limits_{k=1}^N a_k e^{2\pi  w_k \left( \xi + {{N-1} \over {2N}} \right)} &  \ldots  & (-1)^{N-1} \sum\limits_{k=1}^N a_k e^{2\pi \xi w_k}  e^{2\pi w_k {{N-1} \over {2N}}}  \cdot \prod\limits_{l \ne k} e^{2\pi w_l} \\
 \ldots &    &  \ldots\\
\sum\limits_{k=1}^N a_k e^{2\pi \xi w_k}  & \ldots  & (-1)^{N-1} \sum\limits_{k=1}^N a_k e^{2\pi \xi w_k}  \cdot \prod\limits_{l \ne k} e^{2\pi w_l} \\
\end{pmatrix}$$

{\bf Step 7}. Set $\alpha = 2N$ in Lemma \ref{lem:det}. If the blocks are {\it of the second type}, we also use the change of variables $\xi + {{N-1} \over {2N}} \mapsto \xi$. For simplicity, assume that all the blocks are {\it of the first type}. Note that on the right-hand side of Equation (\ref{eq:det}) only the first product depends on $\xi$. By Lemma \ref{lem:det} we have
$$\det \Delta_m = \pm \prod_{k=1}^N e^{2\pi \xi w_k} \cdot C_2,$$
where the constant $C_2$ depends only on $\{w_k\}_{k=1}^N$, $\{a_k\}_{k=1}^N$ and $N$. Hence,
\begin{equation}\label{eq:det2}
|\det \Delta_m | \ge \prod_{k=1}^N \min_{t \in (0,2)} e^{2\pi t w_k}  \cdot |C_2| =: C_3.
\end{equation}
Combining (\ref{eq:det1}) and (\ref{eq:det2}) we achieve the inequality
$$|\det A_m | = |\det A(\{\lambda_{j_0} \}, \ldots , \{\lambda_{j_{N+N_1}}\})| \ge c > 0$$
for any $m$-th segment of the operator and for some $c>0$, which depends only on $g$ and $N$. \par

{\bf Step 8}.  For any $m \in \Z$ we have $|\det A_m| \ge c >0$ and so there exists $D_m \in M_{N+N_1} (\CC)$ such that $A_m^{-1} = D_m$. It has the form $D_m = {{1} \over {\det A_m}} J_m^T$, where $J_m$ is the adjugate matrix of $A_m$. The collection $\{\|J_m\|\}_{m \in \Z}$ is uniformly bounded by some constant $C_4>0$, since $L_{\xi}$ is bounded. We have
$$\|D_m\| \leq {{1} \over {\sup |\det A_m| }} \cdot C_4 \leq  {{1} \over {c}} \cdot C_4 \text{ for any } m \in \Z.$$


Hence, $L_{\xi}$ has a bounded inverse and so 
$$\|L_{\xi} (x) \| \ge \const \cdot \|x\|_2 \text{ for any } x\in \ell^2.$$

\section{Proof of Theorem 2}\label{sect:Th2}

Fix $\eps>0$ and $M \in \mathbb{N}$, and let $g$ be a function from $\K_1(M)$. Put $M_1 = \left[{{2M} \over {\eps}} \right]+1$ and consider the set $\Lambda = \Lambda(\eps, M)$ defined in Section \ref{sect:Lambda}. Construct an operator $L_{\xi} (g)$ as in Section \ref{sect:Oper} using the functions $m_s(t)$ from Definition \ref{def:m}. For the proof of Theorem 2 we need the following statement.
\begin{lemma}\label{lem:mM}
For any $0\leq t<1$ we have $m_{M-1} (t) \ne 0$.
\end{lemma}
\begin{proof}
According to Definition \ref{def:m} we have to compute $B_{M-1, l}^{(k)} $ explicitly. Note that $\sum_{l\ne k} j_l = M-j_k$. Now for $s=M-1$ we have $B_{M-1, l}^{(k)} = a_{l,j_k-1} A_{k, M-j_k}$, since this is just a coefficient of $e^{2\pi i (M-1)t}$ in Equation (\ref{eq:Pg}). Clearly we obtain
$$A_{k, M-j_k} = \prod_{l \ne k} (-1)^{j_l} \binom{j_l}{j_l}e^{2\pi w_l j_l} = (-1)^{M-j_k} e^{2\pi \cdot \sum_{l\ne k} w_l j_l}, $$ 
since this is a coefficient of $e^{2\pi i (M-j_k)t}$ in Equation (\ref{eq:Pk}).  It remains to note that 
$$a_{l,j_k-1} =e^{2\pi w_k(j_k-1)}  {{(-1)^{j_k-1-l}  (2\pi i)^{j_k-1-l}} \over {(j_k-1)!}} \binom{j_k-1}{l}.$$
Now we obtain
$$m_{M-1}(t) = \sum_{k=1}^N a_k A_{k, M-j_k} e^{2\pi w_k t} \sum_{l=0}^{j_{k-1}} a_{l,j_k-1} (2\pi i)^l t^l =$$
$$= (-1)^{M-1} e^{2\pi \sum_{b=1}^N w_b j_b} \cdot \sum_{k=1}^N a_k e^{2\pi w_k (t-1)}  {{(2\pi i)^{j_k-1}} \over {(j_k-1)!}}  (t-1)^{j_k-1}  \ne 0$$
 for any $0\leq t<1$ because of Definition \ref{def:K2}.
\end{proof}
 Note that in the proof of Lemma \ref{lem:mM} we obtained a direct formula for $m_{M-1}$, which coincides with formula (\ref{eq:mN}) in the case of $j_1 = \ldots = j_N = 1$.\par
The structure of the operator $L_{\xi}(g)$ is the same as in the case of simple poles. Hence, by direct application of Step 2, Step 3 and Step 4 of Section \ref{sect:Th1} we may assume that $L_{\xi}(g)$ consists of  segments $A_m(g)$ of the form 
$$
\newcommand*{\tempb}{\multicolumn{1}{|c}{0}}
\newcommand*{\tempv}{\multicolumn{1}{|c}{\vdots}}
A_m(g) = \begin{pmatrix} 
\star & \star & \ldots & \star & \tempb &  0  & 0&\ldots  &0\\
\star & \star & \ldots & \star  & \tempb  & 0 & 0 & \ldots  & 0\\
\vdots & \vdots&  & \vdots &   \tempv& \vdots  &  \vdots  & & \vdots \\
\star & \star & \ldots & \star   & \tempb & 0 &  0 & \ldots  &0 \\
\cline{1-4}
0 & \star & \star & \ldots & \star  & 0 & 0 & \ldots  &0 \\
0 & 0 & \star & \star & \ldots & \star  & 0 & \ldots  &0\\
\vdots & &   & \ddots &   \ddots & & \ddots& &\vdots  \\
0 & \ldots  & 0 & 0 & \star &  \star  & \ldots& \star & 0\\
0 & \ldots  & 0 & 0 & 0 & \star & \star  & \ldots& \star\\
\end{pmatrix}.$$
Each segment is of size $(M+M_1) \times (M+M_1)$. Combining Step 5 of Section \ref{sect:Th1} and Lemma \ref{lem:mM} we obtain
$$\left| \prod_{j=M+1}^{M+M_1} a_{M-1}(j) \right| \ge C_1^{M_1}$$
for some constant $C_1>0$, which does not depend on the segment. It remains to show that the blocks of $L_{\xi}(g)$ are uniformly invertible.\par
Observe that 
$${{a_k } \over {(t-iw_k)^{j_k}}} = {{a_k} \over {(j_k-1)! i^{j_k-1}}} \cdot {{\partial^{j_k-1} } \over {\partial w^{j_k-1}}} \left({{1}\over {t-iw_k}} \right).$$
 Fix $\eps_1>0$ and construct a function 
$$g_{\eps_1} (t) = \sum_{k=1}^N {{a_k} \over {(j_k-1)! i^{j_k-1} \eps_1^{j_k-1}}} \sum_{l=0}^{j_k-1} (-1)^{j_k -l-1} \binom{j_k-1}{l} {{1} \over {t-i(w_k + l\eps_1)}}.$$
If $\eps_1$ is small enough, then $g_{\eps_1}(t)$ is a rational function with simple poles $i w_{k,l} = i(w_k+l\eps_1)$, where $0\leq l \leq j_k-1$. The number of poles is precisely $M$, so one can construct the operator $L_{\xi}(g_{\eps_1})$ for that function. \par
Denote by $\Delta_m(g)$ (respectively, $\Delta_m(g_{\eps_1})$) the block of the $m$-th segment of $L_{\xi}(g)$ (respectively, $L_{\xi}(g_{\eps_1})$). Applying Steps 2, 3, 4, and 5 from Section \ref{sect:Th1} to $L_{\xi}(g_{\eps_1})$, we may assume that $\Delta_m(g_{\eps_1})$ is a square matrix with precisely $M+M_1$ rows and columns. \par

We want to apply Lemma \ref{lem:det} to $\Delta_m(g_{\eps_1})$ as $\eps_1 \to 0$. Consider two poles $w_{k,l}$ and $w_{k,n}$ of $g_{\eps_1}$. We have
$$e^{-2\pi\frac{ w_{k,n}}{\alpha}} - e^{-2\pi\frac{ w_{k,l}}{\alpha}} \sim e^{-2\pi w_k/\alpha} \cdot \left( \frac{2\pi}{\alpha} \right) \cdot (l - n)\eps_1.$$
So the second product in Lemma \ref{lem:det} gives us $\eps_1^{j_k(j_k-1)/2}$. The third product also gives us $\eps_1^{j_k(j_k-1)/2}$, while the coefficients in the denominator give us ${{1} \over {\eps_1^{j_k(j_k-1)}}}$. It remains to note that for the poles $w_{k,l}$ and $w_{i,j}$ for $k \ne i$ their differences tend to a constant as $\eps_1 \to 0$. Hence, combining Steps 6 and 7 of Section \ref{sect:Th1}, we obtain
$$|\det \Delta_m (g_{\eps_1})| \ge  C_2 \quad \text{for any } m \in \Z$$
for sufficiently small $\eps_1$ and some constant $C_2>0$ depending only on the original function $g$. Now we only need to show that
$$\lim_{\eps_1 \to 0} |\det \Delta_m(g_{\eps_1})|  = |\det \Delta_m(g)| \quad \text{pointwise for any } m \in \Z.$$
Indeed, $\det \Delta_m(g)$ is a polynomial function of the variables $m_j$ in the sense of Definition \ref{def:m}, while $\det \Delta_m(g_{\eps_1})$ is {\it the same function} of the variables $m_j$, which are the functions {\it of the same arguments} in the sense of Equation (\ref{eq:m}). By the definition of $m_s$, it is sufficient to prove the following lemma.
\begin{lemma} For the function $g_{\eps_1}$ define the functions $M_s^{\eps_1}(t)$ by \eqref{eq:Mst}. Then for any $s$ we have $M_s^{\eps_1}(t) \to M_s(t)$ for $t \in [0,1]$.
\end{lemma}
\begin{proof} For $P(t) = \prod_{k=1}^M (1 - e^{2\pi i(t-iw_k)})^{j_k}$ we have
$$P(t) \sum_{n \in \mathbb{Z}} c_{m,n} g(t-n) = \sum_{s=0}^{M-1} M_s(t) e^{2\pi i s t}$$
by definition of $M_s(t)$. Also, for $P_{\eps_1}(t) = \prod_{k=1}^M \prod_{l=0}^{j_k-1} (1 - e^{2\pi i(t - iw_{k,l})})$ we have 
$$P_{\eps_1}(t) \sum_{n \in \mathbb{Z}} c_{m,n} g_{\eps_1}(t-n) = \sum_{s=0}^{M-1} M_s^{\eps_1}(t) e^{2\pi i s t}$$
by definition of $M_s^{\eps_1}(t)$. Since $w_{k,l} \to w_k$ as $\eps_1 \to 0$, we obtain that $P_{\eps_1} $ converges to $P$ uniformly on any compact set. Also, $g_{\eps_1}$ converges to $g$ as $\eps_1 \to 0$ pointwise (actually, uniformly away from the poles). Hence, 
$$\sum_{s=0}^{M-1} M_s^{\eps_1}(t) e^{2\pi i s t} \xrightarrow{\eps_1 \to 0} \sum_{s=0}^{M-1} M_s(t) e^{2\pi i s t} \quad \text{pointwise}.$$
The functions $\{e^{2\pi i st}\}_{s=0}^{M-1}$ are linearly independent, so we have $M_s^{\eps_1} \to M_s$ as $\eps_1 \to 0$ pointwise.
\end{proof}
Now for any $m \in \Z$ one can obtain
$$| \det A_m(g) | =  |\det \Delta_m (g)| \cdot \prod_{k=M+1}^{M+M_1+1} |a_{M-1} (k)| \ge C_1^{M_1} C_2 > 0,$$
Hence, $L_{\xi} (g)$ has a bounded inverse for almost all $\xi \in [0,1)$ by Step 8 of Section \ref{sect:Th1}, which completes the proof.

\end{document}